\documentclass[a4paper,12pt]{article}
\usepackage[utf8]{inputenc}
\title{Superconvergence of a Galerkin FEM for Higher-Order Elements in Convection-Diffusion Problems}
\author{Sebastian Franz\footnote{
           Institut f\"ur Numerische Mathematik, Technische Universit\"at Dresden,
           01062 Dresden, Germany.
           \mbox{e-mail}: \{sebastian.franz, hans-goerg.roos\}@tu-dresden.de}\and
        Hans-G\"org Roos\footnotemark[1]
        }
\date{\today}

\usepackage{fancyhdr} 
\usepackage{pgfpages}

\usepackage{amsmath}
\usepackage{amssymb}
\usepackage{amsthm}

\usepackage{cite}

\usepackage{mathptmx}      

\newcommand{\ord}[1]{\mathcal{O}\left(#1\right)}
\DeclareMathOperator{\meas}{meas}

\newcommand{\laplace}{\Delta}
\newcommand{\grad}{\nabla}

\newcommand{\eps}{\varepsilon}
\newcommand{\norm}[2]{\|{#1}\|_{#2}}

\newcommand{\bignorm}[2]{\left\| {#1} \right\|  _{#2}}
\newcommand{\tnorm}[1]{\left|\!\!\;\left|\!\!\;\left| {#1}
                       \right|\!\!\;\right|\!\!\;\right|}
\newcommand{\enorm}[1]{\tnorm{#1}_\eps}

\newcommand{\half}{\frac{1}{2}}

\newcommand{\PS}{\mathcal{P}}
\newcommand{\QS}{\mathcal{Q}}

\newcounter{tmp}
\newcommand{\makeballnumber}[1]{\setcounter{tmp}{\theenumi}%
\setcounter{enumi}{#1}%
\leavevmode \csname beamer@@tmpl@enumerate item\endcsname%
\setcounter{enumi}{\thetmp}}
\newcommand{\makeball}{\leavevmode \csname beamer@@tmpl@itemize item\endcsname}

\numberwithin{equation}{section}

\DeclareMathAlphabet{\mathcal}{OMS}{cmsy}{m}{n}

\theoremstyle{plain}
\newtheorem{thm}{Theorem}[section]
\newtheorem{lem}[thm]{Lemma}
\newtheorem{ass}[thm]{Assumption}

\newtheorem{cor}[thm]{Corollary}
\newtheorem{rem}[thm]{Remark}

\begin{document}
  \thispagestyle{empty}

  \mbox{}\hfill Als Manuskript gedruckt

  \vfill

  \begin{center}
    Technische Universit\"at Dresden \\
    Herausgeber: Der Rektor

    \vfill
    \vfill

    \begin{large}
       \makeatletter
       \textbf{\@title}
       \makeatother
    \end{large}
    \bigskip

    Sebastian Franz und Hans-Görg Roos

    \bigskip

     MATH-NM-03-2013

    \medskip

     \today

  \end{center}

  \vspace{5cm}

  \mbox{}

  \newpage
  \thispagestyle{empty}
  \mbox{}
  \newpage

 \pagestyle{fancy}
  \maketitle
  \begin{abstract}
     In this paper we present a first supercloseness analysis for
     higher-order Galerkin FEM applied to a singularly perturbed
     convection-diffusion problem. Using a solution decomposition
     and a special representation of our finite element space
     we are able to prove a supercloseness property of $p+1/4$
     in the energy norm where the polynomial order $p\geq 3$ is odd.
  \end{abstract}



 \section{Introduction}
  Consider the convection dominated convection-diffusion problem
  \begin{subequations}\label{eq:Lu}
  \begin{align}
   -\eps\laplace u-(b\cdot\grad)u+cu&=f,\quad\mbox{in }\Omega=(0,1)^2\\
    u&=0,\hspace*{0.4cm}\mbox{on }\partial\Omega
  \end{align}
  \end{subequations}
  where $c\in L_\infty(\Omega)$, $b\in W_\infty^1(\Omega)$,
  $f\in L_2(\Omega)$ and $0<\eps\ll1$,
  assuming
  \begin{gather}\label{eq:usual}
    c+\half b_x\geq\gamma>0.
  \end{gather}
  For a problem with exponential layers, i.e. in the case $b_1(x,y)\geq\beta_1>0$, $b_2(x,y)\geq\beta_2>0$,
  we have for linear or bilinear elements in the so called energy norm 
  \[
   \enorm{v}^2:=\eps\norm{\grad v}{0}^2+\norm{v}{0}^2
  \]
  where $\norm{\cdot}{0}$ denotes the usual $L_2$-norm, on a
  Shishkin mesh (for the exact definition see Section~\ref{sec:mmsd})
  \[
   \enorm{u-u^N}\lesssim N^{-1}\ln N.
  \]
  We use the notation $a\lesssim b$, if a generic constant $C$ independent of $\eps$ and $N$
  exists with $a\leq C b$.

  However, for bilinear elements Zhang~\cite{Zhang03} and Linß~\cite{Linss00} observed a supercloseness
  property: the difference between the Galerkin solution $u^N$ and the standard piecewise
  bilinear interpolant $u^I$ of the exact solution $u$ satisfies
  \[
   \enorm{u^I-u^N}\lesssim(N^{-1}\ln N)^{2}.
  \]

  Supercloseness is a very important property. It allows optimal error estimates in $L_2$
  (Nitsche's trick cannot be applied), improved error estimates in $L_\infty$ inside the layer regions
  and recovery procedures for the gradient, important in a posteriori error estimation.
  
  In the last ten years supercloseness for bilinear elements was also proved for problems with
  characteristic layers~\cite{FrL06}, for S-type meshes~\cite{Linss00}, for Bakhvalov meshes~\cite{RoosSchopf10}
  and for several stabilisation methods, including streamline diffusion FEM (SDFEM),
  continuous interior penalty FEM (CIPFEM),
  local projection stabilisation FEM (LPSFEM)
  and discontinuous Galerkin (see e.g. \cite{FrLR06,Fr07,FrM10, RoosZar07,ST03,FrLRS08,Zarin09}).
  Recently, even corner singularities were included in the analysis~\cite{LR13}.
  
  For $\QS_p$-elements with $p\geq2$ the situation is very different.
  Using the so-called vertex-edge-cell interpolant $\pi u$ \cite{Lin91,GR86} instead of
  the standard Lagrange-interpolant with equidistant interpolation points,
  Stynes and Tobiska \cite{ST08} proved for SDFEM (but not for the Galerkin FEM)
  \[
   \enorm{\pi u-\tilde u^N}\lesssim N^{-(p+1/2)},
  \]
  where $\tilde u^N$ denotes the SDFEM solution. It is not clear whether this estimate is optimal.
  The numerical results of \cite{Fr12, Fr13_1} indicate for the Galerkin FEM and
  $p\geq 3$ a supercloseness property of order $p+1$ for two different interpolation
  operators. One of them is the vertex-edge-cell interpolator $\pi u$, the other
  one is the Gauss-Lobatto interpolation operator $I^Nu$.
  For SDFEM, the order $p+1$ is observed numerically for all $p\geq 2$.
  
  In the present paper we study the Galerkin FEM for odd $p$.
  We shall prove some supercloseness properties, but the achieved order is probably not optimal.
  
  The paper is organised as follows. In Section~\ref{sec:mmsd} we provide 
  descriptions of the underlying mesh, the numerical method and a solution 
  decomposition. The main part is Section~\ref{sec:analysis} where the proof 
  of our assertion can be found. As the proof is rather technical we provide
  it in full only for $p=3$ and demonstrate its generalisation for arbitrary
  odd $p\geq 5$. We omit numerical simulations and refer to the results given 
  in~\cite{Fr12, Fr13_1} that show for \emph{any} $p\geq 3$ a supercloseness for the 
  Galerkin method of order $p+1$.

 \section{Mesh, Method and a Solution Decomposition}\label{sec:mmsd}

  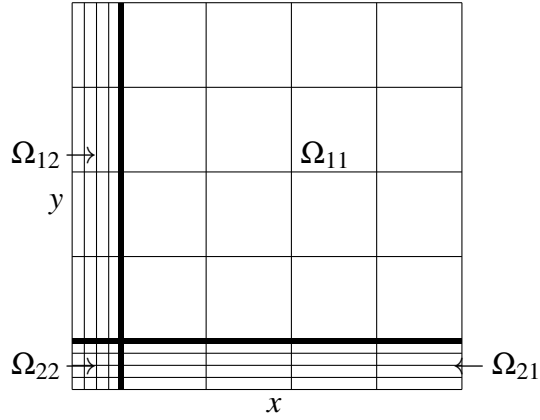
\begin{figure}
   \begin{center}
     \setlength{\unitlength}{0.57pt}
      \begin{picture}(256,256)
       \put(  0,  0){\line( 0, 1){256}}
       \put(  8,  0){\line( 0, 1){256}}
       \put( 16,  0){\line( 0, 1){256}}
       \put( 24,  0){\line( 0, 1){256}}
       \put( 88,  0){\line( 0, 1){256}}
       \put(144,  0){\line( 0, 1){256}}
       \put(200,  0){\line( 0, 1){256}}
       \put(256,  0){\line( 0, 1){256}}

       \put(  0,256){\line( 1, 0){256}}
       \put(  0,200){\line( 1, 0){256}}
       \put(  0,144){\line( 1, 0){256}}
       \put(  0, 88){\line( 1, 0){256}}
       \put(  0, 24){\line( 1, 0){256}}
       \put(  0, 16){\line( 1, 0){256}}
       \put(  0,  8){\line( 1, 0){256}}
       \put(  0,  0){\line( 1, 0){256}}

       \linethickness{2pt}
       \put( 32,  0){\line( 0, 1){256}}
       \put(  0, 32){\line( 1, 0){256}}
       \put(128,-15){$x$}
       \put(-15,120){$y$}
       \put(150,150){$\Omega_{11}$}
       \put(250, 10){$\leftarrow\Omega_{21}$}
       \put(-40,150){$\Omega_{12}\!\to$}
       \put(-40, 10){$\Omega_{22}\!\to$}
      \end{picture}
   \end{center}
   \caption{\label{fig:mesh}Shishkin mesh for Problem \eqref{eq:Lu}}
  \end{figure}

  We discretise the domain by a Shishkin mesh. Under the assumption
  \[
   \eps\leq\frac{\min\{\beta_1,\beta_2\}}{2\sigma\ln N}
  \]
  we define the mesh-transition points by
  \[
    \lambda_x:=\frac{\sigma\eps}{\beta_1}\ln N,\qquad
    \lambda_y:=\frac{\sigma\eps}{\beta_2}\ln N,
  \]
  where $\sigma\geq p+3/2$ is a user-chosen parameter.
  Let $\Omega_{11}=[\lambda_x,1]\times[\lambda_y,1]$,
      $\Omega_{12}=[0,\lambda_x]\times[\lambda_y,1]$,
      $\Omega_{21}=[\lambda_x,1]\times[0,\lambda_y]$,
  and $\Omega_{22}=[0,\lambda_x]\times[0,\lambda_y]$.
  The domain $\Omega$ is dissected by a tensor product mesh $T^N$,
  according to
  \begin{align*}
   x_i&:=\begin{cases}
          \frac{\sigma\eps}{\beta_1}\ln N\frac{2i}{N},
                                          &i=0,\dots,N/2,\\
          1-2(1-\lambda_x)(1-\frac{i}{N}),&i=N/2,\dots,N,
         \end{cases}\\
   y_j&:=\begin{cases}
          \frac{\sigma\eps}{\beta_2}\ln N\frac{2j}{N},
                                          &j=0,\dots,N/2,\\
          1-2(1-\lambda_y)(1-\frac{j}{N}),&j=N/2,\dots,N,
         \end{cases}.
  \end{align*}
  Figure~\ref{fig:mesh} shows an example of $T^N$ for \eqref{eq:Lu}.
  By $h_i$ and $k_j$ we denote the mesh sizes of a specific element
  $\tau_{ij}\in T^N$ in $x$- and $y$- direction, resp.

  Our finite-element space $V^N\subset H_0^1(\Omega)$ on $T^N$ is given by
  \[
     V^N:=\{v\in H_0^1(\Omega):\,v|_\tau\in\QS_p(\tau),\,\forall\tau\in T^N\},
  \]
  where $H_0^1(\Omega)$ is the standard Sobolev space
  $H_0^1(\Omega)=\{v\in H^1(\Omega): v|_{\partial\Omega}=0\}$
  with $v|_{\partial\Omega}=0$ being understood in the sense of traces
  and $\QS_p(\tau)$ is the space of polynomials of degree at most $p$
  in each coordinate direction.

  Then the Galerkin method can be written as:
  Find $u^N\in V^N$ such that
  \[
    a_{Gal}(u^N,v^N)=(f,v^N),\quad\mbox{for all }v^N\in V^N,
  \]
  where the bilinear form $a(\cdot,\cdot)$ is given by
  \[
   a_{Gal}(v,w):=\eps(\grad v,\grad w)+(cv-b\cdot\grad v,w),\quad\mbox{for all }v,\,w\in H_0^1(\Omega),
  \]
  and $(\cdot,\cdot)$ is the standard $L_2$-product in $\Omega$.

  Our analysis is based on a solution decomposition of $u$, which we provide here.
  
  \begin{ass}\label{ass:dec}
   The solution $u$ of problem~\eqref{eq:Lu} can be decomposed as
   \begin{gather*}
      u =S+E_{12}+E_{21}+E_{22},
   \end{gather*}
   where we have for all $x,y\in[0,1]$ and $0\le i+j\le p+2$ the
   pointwise estimates
   \begin{equation}\label{eq:dec:C0}
    \left.
      \begin{aligned}
        \left|\frac{\partial^{i+j} S}{\partial x^i \partial y^j}(x,y)\right|
             &\le C,
        \quad
        \left|\frac{\partial^{i+j} E_{12}}{\partial x^i \partial y^j}(x,y)\right|
              \lesssim \eps^{-i}e^{-\beta_1 x/\eps},\\[0.2cm]
        \left|\frac{\partial^{i+j} E_{21}}{\partial x^i \partial y^j}(x,y)\right|
             &\lesssim \eps^{-j}e^{-\beta_2 y/\eps},\\[0.2cm]
        \left|\frac{\partial^{i+j} E_{22}}{\partial x^i \partial y^j}(x,y)\right|
             &\lesssim \eps^{-(i+j)} e^{-\beta_1 x/\eps}e^{-\beta_2 y/\eps}.
      \end{aligned}
    \right\}
   \end{equation}
   Here $E_{12}$ and $E_{21}$ are exponential boundary layers, $E_{22}$ is a the
   corner layer, and $S$ is the regular part of the solution.
  \end{ass}

  For conditions that guarantee the existence of such a decomposition,
  see \cite[Theorem III.1.26]{RST08}.

  \begin{rem}
   With Assumption~\ref{ass:dec} for $i+j\leq p+1$ we immediately have for $\PS_p$-
   or $\QS_p$-elements
   \[
    \enorm{u-u^N}\lesssim(N^{-1}\ln N)^p.
   \]
   For $\QS_p$-elements this result follows from the proof given in \cite{ST08}
   for the streamline-diffusion FEM.
  \end{rem}

 \section{Supercloseness Analysis}\label{sec:analysis}

  Before we start the analysis, let us define the two interpolation operators $\pi u$
  and $I^N u$ precisely.
  Let $\hat{a}_i$ and $\hat{e}_i$, $i=1,\ldots,4$, denote the vertices and
  edges of the reference element $\hat{\tau}=[-1,1]^2$, respectively.
  We define the \emph{vertex-edge-cell interpolation} operator
  $\hat{\pi}:C(\hat{\tau})\to\QS_p(\hat{\tau})$ by
  \begin{subequations}\label{eq:def:lin}
  \begin{alignat}{2}
   \hat{\pi} \hat{v}(\hat{a}_i)&=\hat{v}(\hat{a}_i),\,\quad i=1,\dots,4,
   &&\label{eq:general:def_inter1}\\
   \int_{\hat{e}_i}(\hat{\pi}\hat{v})\hat{q} &= \int_{\hat{e}_i} \hat{v} \hat{q},
   \quad i=1,\dots,4,\quad
   &&\hat{q}\in \PS_{p-2}(\hat{e}_i),\label{eq:general:def_inter2}\\
   \iint_{\hat{\tau}} (\hat{\pi}\hat{v})\hat{q} &= \iint_{\hat{\tau}} \hat{v} \hat{q},
   &&\hat{q}\in \QS_{p-2}(\hat{\tau}).
                                            \label{eq:general:def_inter3}
  \end{alignat}
  \end{subequations}
  This operator is uniquely defined and can be extended to the globally
  defined interpolation operator
  $\pi^N:C(\overline{\Omega})\to V^N$ by
  \[
   (\pi^N v)|_\tau := \big(\hat{\pi}(v\circ F_\tau)\big)\circ F_\tau^{-1}
   \quad\forall\tau\in T^N,\,v\in
   C(\overline{\Omega}),
  \]
  with the bijective reference mapping $F_\tau:\hat{\tau}\to\tau$.

  Let $-1=t_0<t_1<\dots<t_{p-1}<t_p=+1$ be the zeros of
  \[
    (1-t^2)L_{p}'(t)=0,\quad t\in[-1,1],
  \]
  where $L_p$ is the Legendre polynomial of degree $p$, normalised to $L_p(1)=1$.
  These points are also used in the Gau\ss-Lobatto
  quadrature rule of approximation order $2p-1$. Therefore, we refer
  to them as Gau\ss-Lobatto points.
  We define the \emph{Gauß-Lobatto interpolation} operator
  $\mathcal{I}:C(\hat{\tau})\to\QS_p(\hat{\tau})$ by values at
  \begin{align}\label{eq:def:GL}
     (\mathcal{I}\hat{v})(t_i,t_j) & := \hat{v}(t_i,t_j)
  \end{align}
  and extend it to the operator $I^N:C(\overline{\Omega})\to V^N$ in the same way as above.

  \begin{lem}
    For the interpolation operators $\pi^N:C(\overline{\Omega})\rightarrow V^N$
    and $I^N:C(\overline{\Omega})\rightarrow V^N$
    holds the stability property
   \begin{align}
    \bignorm{\pi^N w}{L_\infty(\tau)}+
    \bignorm{I^N w}{L_\infty(\tau)}
    &\lesssim \norm{w}{L_\infty(\tau)}\quad\forall w\in C(\tau),\,
    \forall \tau\subset\overline{\Omega},\label{eq:inter:prop_stab}
   \end{align}
   and for $\tau_{ij}\subset\overline{\Omega}$ and $q\in[1,\infty]$,
   $2\leq s\leq p+1$, $1\leq t\leq p$ hold the anisotropic error estimates
   \begin{subequations}\label{eq:inter:aniso}
   \begin{align}
    \bignorm{w-\pi^N w}{L_q(\tau_{ij})}+
    \bignorm{w-I^N w}{L_q(\tau_{ij})}
     &\lesssim \sum_{r=0}^s h_i^{s-r}k_j^r
                \bignorm{\frac{\partial^{s}w}
                              {\partial x^{s-r}\partial y^{r}}}
                        {L_q(\tau_{ij})},\label{eq:inter:aniso:1}\\
    \bignorm{(w-\pi^N w)_x}{L_q(\tau_{ij})}+
    \bignorm{(w-I^N w)_x}{L_q(\tau_{ij})}
     &\lesssim \sum_{r=0}^t h_i^{t-r}k_j^r
                \bignorm{\frac{\partial^{t+1}w}
                              {\partial x^{t-r+1}\partial y^{r}}}
                        {L_q(\tau_{ij})}
   \end{align}
   \end{subequations}
   and similarly for the $y$-derivative.
  \end{lem}
  \begin{proof}
   The proof can be found in \cite{Apel99,ST08,FrM10_1}.
  \end{proof}

  \begin{lem}\label{lem:inter}
   For the interpolation operators $\pi^N:C(\overline{\Omega})\rightarrow V^N$
   and $I^N:C(\overline{\Omega})\rightarrow V^N$ we have the interpolation
   error results
   \begin{subequations}
   \begin{align}
    \norm{u-\pi u}{0}+\norm{u-I^Nu}{0}&\lesssim (N^{-1}\ln N)^{p+1}\label{eq:inter_L2}\\
    \enorm{u-\pi u}+\enorm{u-I^Nu}&\lesssim (N^{-1}\ln N)^p\label{eq:inter:eps}
   \end{align}
   \end{subequations}
  \end{lem}
  \begin{proof}
    The proof can be found in \cite{Apel99,ST08,FrM10_1}.
  \end{proof}

  Let us come to the supercloseness analysis and denote by $J^Nu\in V^N$ some
  interpolation of $u$. Then the analysis is based on a standard arguments
  involving coercivity and Galerkin orthogonality and yields
  \begin{gather}\label{eq:analysis:start}
    \enorm{J^Nu-u^N}^2\lesssim a_{Gal}(J^Nu-u^N,J^Nu-u^N)=-a_{Gal}(u-J^Nu,\chi)
  \end{gather}
  where $\chi:=J^Nu-u^N\in V^N$. Thus one has to estimate
  \begin{subequations}
   \begin{align}
    &\eps(\grad(u-J^Nu),\grad\chi)\label{eq:analysis:I}\\
    &(b\cdot\grad(u-J^Nu),\chi)\label{eq:analysis:IIa}
    \mbox{ or equivalently using integration by parts }
     (u-J^Nu,b\cdot\grad\chi)\\
    &(c(u-J^Nu),\chi)\label{eq:analysis:IIIa}
    \mbox{ or if integration by parts was used }
    ((c-\mbox{div}b)(u-J^Nu),\chi)
   \end{align}
  \end{subequations}

  \begin{lem}\label{lem:1}
    It holds
   \begin{gather}\label{eq:analysis:1}
   |(c(u-J^Nu),\chi)|\lesssim(N^{-1}\ln N)^{p+1}\enorm{\chi}
   \end{gather}    
  \end{lem}
  \begin{proof}
    Assuming $J^N$ to be any of our two interpolation operators $\pi^N$ or $I^N$,
    the $L_2$ interpolation error estimate \eqref{eq:inter_L2} yields for the
    reaction term \eqref{eq:analysis:IIIa}
    \[
     |(c(u-J^Nu),\chi)|\leq \norm{c}{L_\infty(\Omega)}
                           \norm{u-J^Nu}{0}
                           \norm{\chi}{0}
                      \lesssim(N^{-1}\ln N)^{p+1}\enorm{\chi}
    \]
    and similarly for the term involving $c-\mbox{div} b$.
  \end{proof}
  \begin{lem}\label{lem:2}
   It holds
  \begin{align}
    |\eps(\grad(u-\pi^Nu),\grad\chi)|&\lesssim N^{-(p+1)}\enorm{\chi},\label{eq:analysis:2}\\
    |\eps(\grad(u-  I^Nu),\grad\chi)|&\lesssim (N^{-1}\ln N)^{p+1}\enorm{\chi},\label{eq:analysis:3}
  \end{align}
  \end{lem}
  \begin{proof}
   In the case of the vertex-edge-cell interpolation operator $\pi^Nu$ we find in
   \cite[Lemma 10]{ST08} the estimate 
   \[
     |\eps(\grad(u-\pi^Nu),\grad\chi)|\lesssim N^{-(p+1/2)}\enorm{\chi}.
   \]
   A close inspection of the proof shows, that the only limiting
   term comes from \cite[(3.16)]{ST08}
   \[
    N^{1/2}\norm{\pi^NE_{22}}{0,\Omega_{12}\cup\Omega_{21}}
    \lesssim (\eps(\eps+N^{-1}\ln N))^{1/2}N^{-(\sigma-1/2)}
    \lesssim (\eps(\eps+N^{-1}\ln N))^{1/2}N^{-(p+1/2)}
   \]
   because $\sigma\geq p+1$ was chosen in \cite{ST08}.
   All other terms involved are of order $p+1$.
   In our paper we have $\sigma\geq p+3/2$,
   and therefore \eqref{eq:analysis:2} follows.

   For the Gauß-Lobatto interpolation operator $I^N$
   we denote by a subscript the polynomial order of the interpolation,
   i.e. we write $I^N_p$ and $\pi^N_p$ for the interpolation operators projecting
   into the FEM-spaces of order $p$.

   In \cite{Fr12} we find the identity
   \[
     I^N_p u = \pi^N_p u+\left(I^N(u-\pi^N_{p+1}u)-(u-\pi^N_{p+1}u))\right)+(u-\pi^N_{p+1}u)
   \]
   also written as
   \begin{gather}\label{eq:identity}
     I^N_p u = \pi^N_p u+R u+(u-\pi^N_{p+1}u)
   \end{gather}
   where $Ru:= I^N(u-\pi^N_{p+1}u)-(u-\pi^N_{p+1}u))$. These are consequences of the basic
   identity
   \[
     \pi^N_p=I^N_p\pi^N_{p+1}.
   \]
   We apply \eqref{eq:identity} to the diffusion term \eqref{eq:analysis:I} and obtain
   \[
    \eps|(\grad(u-I_p^Nu),\grad\chi)|
     \leq  \eps|(\grad(u-\pi_p^N u),\grad\chi)|
          +\eps|(\grad(u-\pi_{p+1}^Nu),\grad\chi)|
          +\eps|(\grad Ru,\grad\chi)|.
   \]
   Now \eqref{eq:analysis:2}, the interpolation error result \eqref{eq:inter:eps}
   for $p+1$ and \cite[Theorem 4.4]{Fr12}, i.e.
   \[
     \eps^{1/2}\norm{\grad Ru}{0}\lesssim (N^{-1}\ln N)^{p+1}
   \]
   prove \eqref{eq:analysis:3}.
  \end{proof}

  What is left is the convective term \eqref{eq:analysis:IIa} and we will analyse it
  for the Gauß-Lobatto interpolation operator $I^N$. This estimate is the crucial point
  of the analysis. Stynes and Tobiska \cite[Remark 16]{ST08} state that the so called
  Lin-identities of \cite{Lin91,Yan08} do not yield bounds of order $p+1$. Instead, they use
  a fairly standard trick in the analysis of stabilised methods to obtain the order $p+1/2$
  for the streamline-diffusion method and the vertex-edge-cell interpolation operator $\pi^N$.
  
  \begin{lem}\label{lem:3}
     It holds for any boundary layer function $E$ of our decomposition $u=S+E_1+E_2+E_{12}$
     \begin{gather}\label{eq:analysis:4}
      |(E-I^NE,b\cdot\grad\chi)|
         \lesssim (N^{-1}\ln N)^{p+1}\enorm{\chi}.
     \end{gather}   
  \end{lem}
  \begin{proof}
    We will make use of the anisotropic interpolation error bounds
    \eqref{eq:inter:aniso:1} and derive
     \begin{align*}
      \norm{E_{12}-I^NE_{12}}{0,\Omega_{12}\cup\Omega_{22}}^2
       &\lesssim \sum_{\tau_{ij}\subset\Omega_{12}\cup\Omega_{22}}
                \sum_{r=0}^{p+1} h_i^{s-r}k_j^{r}
                 \bignorm{\frac{\partial^{s}E_{12}}
                               {\partial x^{s-r}\partial y^{r}}}
                         {0,\tau_{ij}}^2\notag\\
       &\lesssim \sum_{r=0}^{p+1}
                 (\eps N^{-1}\ln N)^{2(s-r)}
                 N^{-2r}
                 \eps^{2(r-s)}
                 \bignorm{e^{-\beta_1 x/\eps}}
                         {0,\Omega_{12}\cup\Omega_{22}}^2\notag\\
       &\lesssim \eps(N^{-1}\ln N)^{2(p+1)}
      \intertext{while ideas from \cite[Lemma 9]{ST08} help us with}
      \norm{E_{12}-I^NE_{12}}{0,\Omega_{11}}
       &\lesssim \norm{E_{12}}{0,\Omega_{11}}
                +\norm{I^NE_{12}}{0,\Omega_{11}}\notag\\
       &\lesssim \eps^{1/2}N^{-\sigma}+(\eps^{1/2}+N^{-1/2})N^{-\sigma}
        \lesssim (\eps^{1/2}+N^{-1/2})N^{-\sigma}
      \intertext{and finally a Hölder inequality, stability \eqref{eq:inter:prop_stab} and
                 $\meas(\Omega_{21})\lesssim\eps\ln N$ yields}
      \norm{E_{12}-I^NE_{12}}{0,\Omega_{21}}
       &\lesssim \meas(\Omega_{21})^{1/2}\left(\norm{E_{12}}{L_\infty(\Omega_{21})}
                                              +\norm{I^NE_{12}}{L_\infty(\Omega_{11})}\right)\notag\\
       &\lesssim \eps^{1/2}(\ln N)^{1/2}N^{-\sigma}.
     \end{align*}
%
     Thus, we obtain
     \begin{align*}
      |(E_{12}-I^NE_{12},b\cdot\grad\chi)|
       &\lesssim \eps^{1/2}((N^{-1}\ln N)^{p+1}+N^{-\sigma}(\ln N)^{1/2})\norm{\grad\chi}{0,\Omega}
                +N^{-\sigma-1/2}\norm{\grad\chi}{0,\Omega_{11}}\\
       &\lesssim (N^{-1}\ln N)^{p+1}\enorm{\chi}
                +N^{-\sigma+1/2}\norm{\chi}{0,\Omega_{11}}
        \lesssim (N^{-1}\ln N)^{p+1}\enorm{\chi}
     \end{align*}
     where $\sigma\geq p+3/2$ and an inverse inequality was used in estimating in $\Omega_{11}$.
     Similarly the other two layer terms can be estimated.
  \end{proof}

  Surprisingly, the real difficulty lies in the estimation of the convective term
  \eqref{eq:analysis:IIa} for the smooth part $S$. The following estimates are
  rather technical. Therefore we split the analysis and start with the one-dimensional
  case and the polynomial order $p=3$. The generalisation into arbitrary odd order $p$ and 2d follows.
  Some ideas of our proof go back 30 years to Axelsson and Gustafsson~\cite{AG81}.
  
  The basic idea is to use a special representation of a piecewise cubic function $v$
  with a basis consisting almost completely of functions that are symmetric
  w.r.t. their domain of support.
  
%
  They are defined on the reference intervals with Legendre 
  polynomials $L_k$ normalised to $L_k(1)=1$.
  We define the standard piecewise linear hat-function
  \begin{align*}
      \hat\phi(t)&:=\frac{1-L_1(2|t|-1)}{2}
                  =1-|t|&
                 &\mbox{for }t\in[-1,1],
  \intertext{a quadratic bubble function}
      \hat\chi_2(t)&:=\frac{1-L_2(2t-1)}{2}
                    =3t(1-t)&
                 &\mbox{for }t\in[0,1]
  \intertext{and a piecewise cubic bubble function}
      \hat\psi_3(t)&:=\frac{L_1(2|t|-1)-L_3(2|t|-1)}{2}
                    = 5 |t|(2|t| - 1) (|t| - 1)&
                 &\mbox{for }t\in[-1,1].
  \end{align*}
%
  Figure~\ref{fig:basis}
  \begin{figure}
      \begin{center}
       \includegraphics[width=0.32\textwidth]{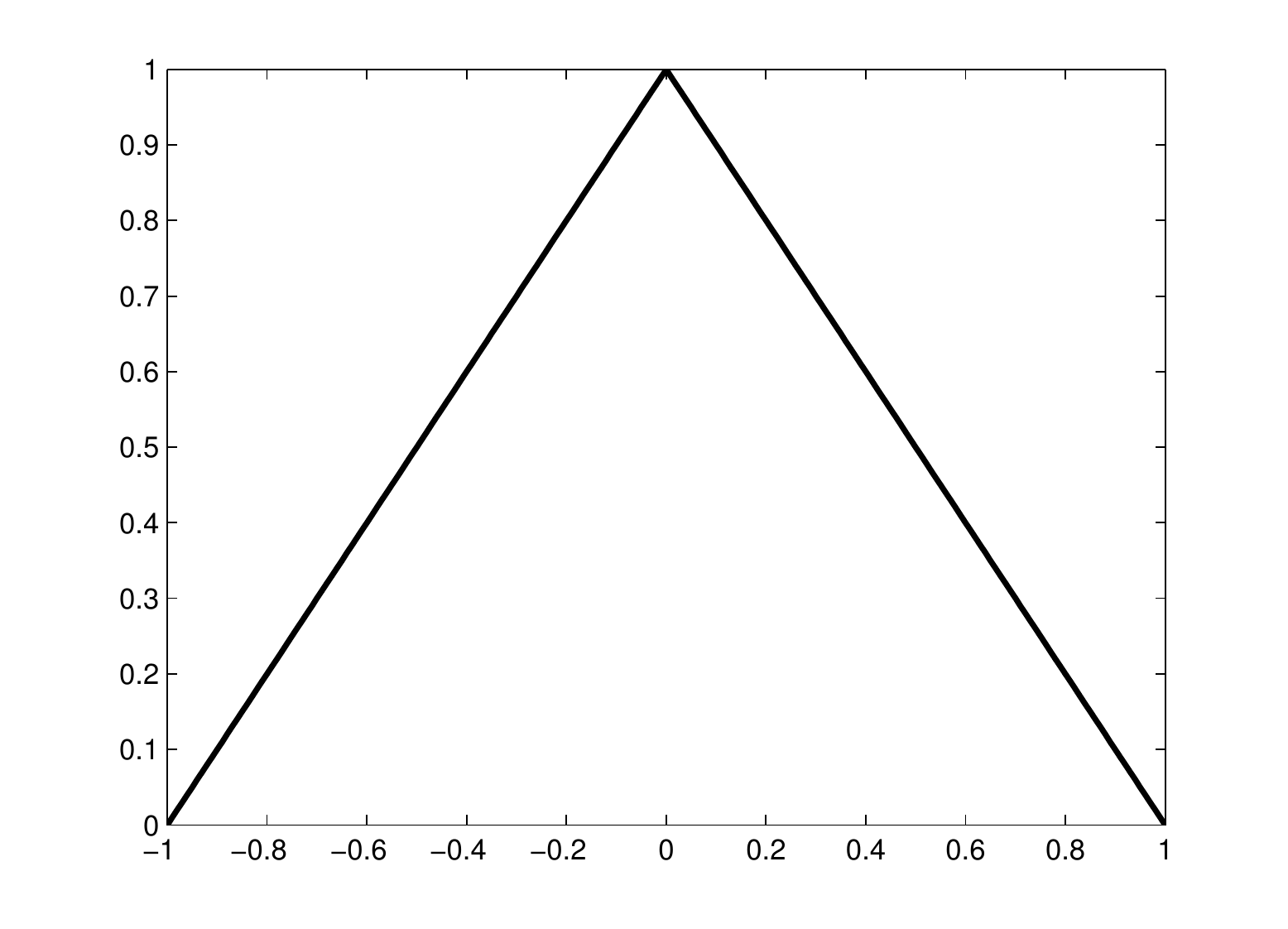}
       \includegraphics[width=0.32\textwidth]{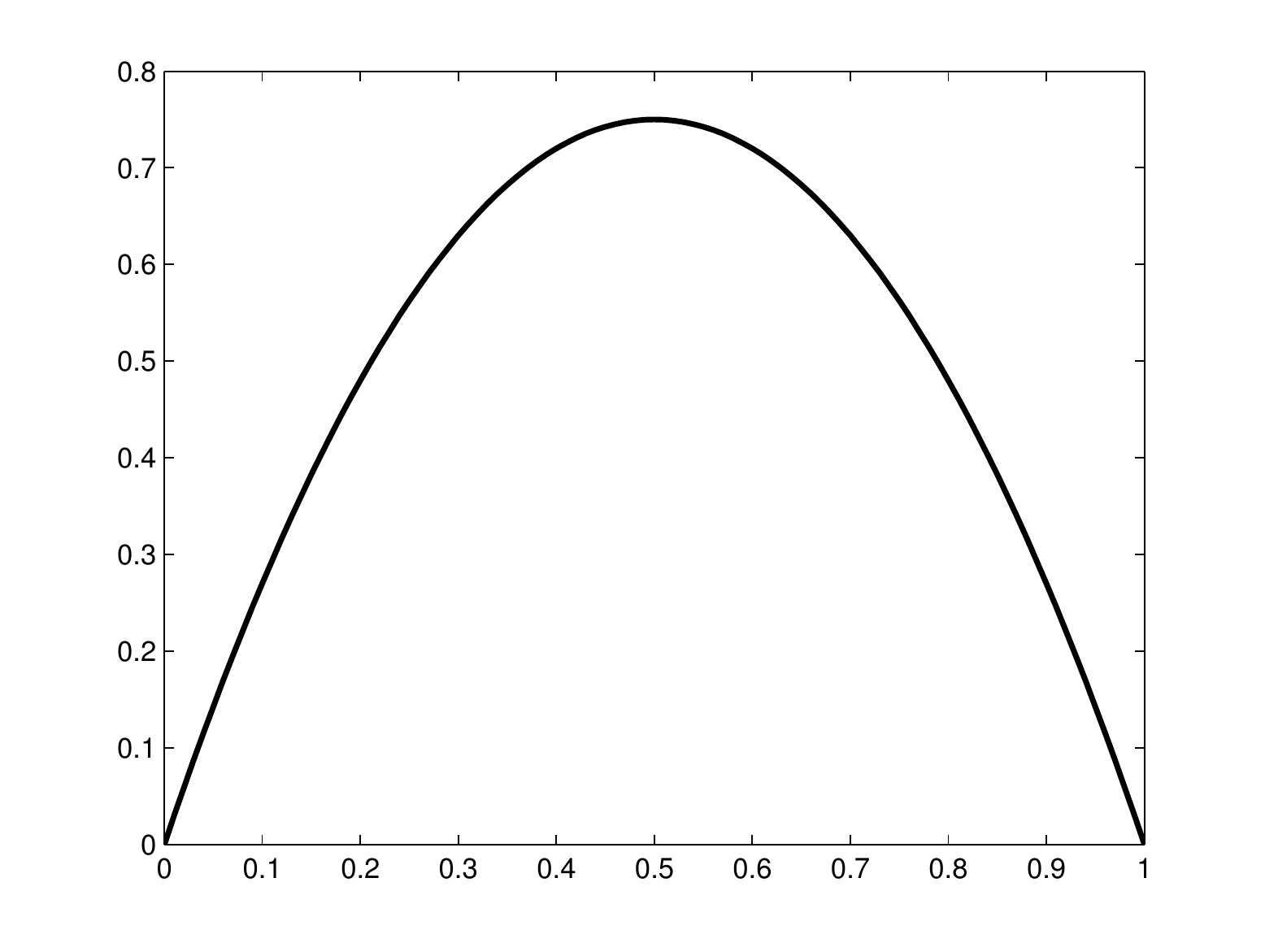}
       \includegraphics[width=0.32\textwidth]{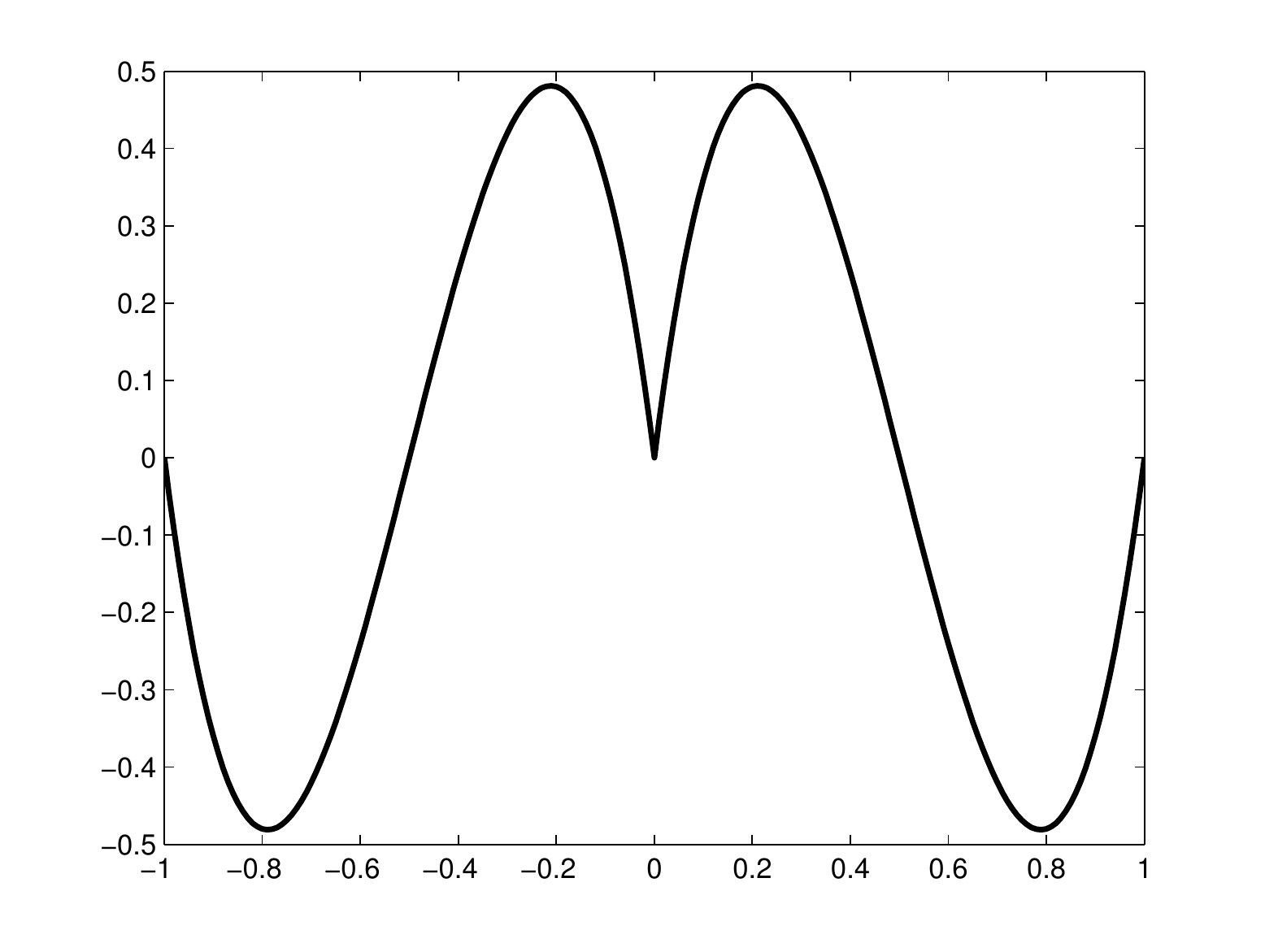}
      \end{center}
      \caption{Basis function $\hat\phi$ (left), $\hat\chi_2$ (middle) and $\hat\psi_3$ (right) on their domains of support
               \label{fig:basis}}
  \end{figure}
  shows the three basis functions on their reference intervals. Let us denote by
  $F_i$ the piecewise linear mapping of $[-1,1]$ onto $[x_{i-1},x_{i+1}]$, such that
  $[-1,0]$ is mapped linearly onto $[x_{i-1},x_i]$ and $[0,1]$ is mapped linearly onto
  $[x_i,x_{i+1}]$. Note that in general the mapping $F_i$ is non-linear.

  Above transformation and the functions on the reference intervals lead to
  the definition of the basis functions
  \begin{align*}
     \phi_i(x)&=\begin{cases}
                   \hat\phi(F_i^{-1}(x)),&x\in[x_{i-1},x_{i+1}]\\
                   0 & \mbox{otherwise}
                \end{cases},&&\,i=1,\dots,N-1,\\
     \chi_{2,i}(x)&=\begin{cases}
                      \hat\chi_2\left(\frac{x-x_{i-1}}{h_i}\right),&x\in[x_{i-1},x_{i}]\\
                      0 & \mbox{otherwise}
                    \end{cases},&&\,i=1,\dots,N,\\
     \psi_{3,i}(x)&=\begin{cases}
                      \hat\psi_3(F_i^{-1}(x)),&x\in[x_{i-1},x_{i+1}]\\
                      0 & \mbox{otherwise}
                     \end{cases},&&\,i=1,\dots,N-1.
  \end{align*}
  Finally, $\psi_{3,N}$ is the left part of $\hat\psi_3$ mapped onto $[x_{N-1},1]$.
  
  Now we obtain for $v$ the representation
  \begin{gather}\label{eq:representation}
       v=\sum_{i=1}^{N-1}(v_i\phi_i+w_i\psi_{3,i})+\sum_{i=1}^Ny_j\chi_{2,j}+w_n\psi_{3,n}.
  \end{gather}
  
  The functions $\phi_i$, $\psi_{3,i}$ and $\chi_{2,j}$ are all symmetric w.r.t.
  their domain of support, with only a few exceptions.
  The last function $\psi_{3,N}$ is antisymmetric on $[x_{N-1},1]$, and
  $\phi_{N/2}$ and $\psi_{3,N/2}$ are in general not symmetric on a
  Shishkin mesh, as here two intervals with different sizes meet.
    
  For a unique representation we still have to define the coefficients in \eqref{eq:representation}.
  We use the following degrees of freedom
  \begin{subequations}
     \begin{align}
       N_1^iv&:=v(x_i),\,i=1,\dots,N-1,\\
       N_2^jv&:=\frac{\int_{x_{j-1}}^{x_j} L_2^j(x) v(x)dx}
                     {\int_{x_{j-1}}^{x_j} L_2^j(x) \chi_{2,j}(x)dx},\,j=1,\dots,N,\\
       N_3^iv&:=\frac{\int_{x_{i-1}}^{x_i} L_3^i(x) v(x)dx}
                     {\int_{x_{i-1}}^{x_i} L_3^i(x) \psi_{3,i}(x)dx},\,i=1,\dots,N
     \end{align}
  \end{subequations}
  where $L_k^i$ is the $k$-th Legendre polynomial $L_k$ mapped onto $[x_{i-1},x_i]$. Then it follows
  \begin{align*}
       v_i &= N_1^iv,&
       y_j &= N_2^jv,&
       w_i &= \int_0^{x_i}\widetilde L_3 v
  \end{align*}
  where
  \[
      \widetilde L_3\big|_{x_{k-1}}^{x_k}=\frac{L_3^k}
                                          {\int_{x_{k-1}}^{x_k} L_3^k(x) \psi_{3,k}(x)dx}.
  \]
  
  With the representation ~\eqref{eq:representation} we can write the $L_2$-norm of $v$ as
  \begin{align*}
     \norm{v}{0}^2
      &= \bignorm{\sum_{i=1}^{N-1}(v_i\phi_i+w_i\psi_{3,i})}{0}^2+
         \bignorm{\sum_{j=1}^Ny_j\chi_{2,j}}{0}^2
         +\norm{w_N\psi_{3,N}}{0}^2\\
      &\quad
         +2\left(\sum_{i=1}^{N-1}v_i\phi_i,\sum_{j=1}^Ny_j\chi_{2,j}\right)
         +2\left(\sum_{i=1}^{N-1}v_i\phi_i,w_N\psi_{3,N}\right).
  \end{align*}
  All other scalar products involve the even functions $\chi_{2,j}$ and the
  functions $\psi_{3,i}$ that are either zero or odd on the support of $\chi_{2,j}$.
  Thus, those scalar products are zero.
  The two remaining scalar products can be
  rewritten as
  \begin{align*}
      \left(\sum_{i=1}^{N-1}v_i\phi_i,\sum_{j=1}^Ny_j\chi_{2,j}\right)
      &=  \sum_{i=1}^{N-1}v_i\left[y_i\int_{x_{i-1}}^{x_i}\phi_i\chi_{2,i}+y_{i+1}\int_{x_i}^{x_{i+1}}\phi_i\chi_{2,i+1}\right]
       = \frac{1}{4}\sum_{i=1}^{N-1}v_i\left[h_iy_i+h_{i+1}y_{i+1}\right]\\
      \left(\sum_{i=1}^{N-1}v_i\phi_i,w_N\psi_{3,N}\right)
      &= v_{N-1}w_N\int_{x_{N-1}}^{x_N}\phi_{N-1}\psi_{3,N}
       = \frac{1}{12}v_{N-1}w_N h_N.
  \end{align*}    

  \begin{lem}\label{lem:4_1d}
   Let $p=3$ and consider the one-dimensional case. Then we obtain for the convective term in the
   smooth part $S$
   \begin{gather}\label{eq:analysis:5}
     \left|\int_0^1 b(S-\hat S)'v\right|
       \lesssim N^{-(3+1/4)}\enorm{v}.
   \end{gather}
  \end{lem}
  \begin{proof}
    Let $\{x_i\}$ be a Shishkin mesh on $[0,1]$, i.e.
    \begin{align*}
      x_i&:=\begin{cases}
             \frac{\sigma\eps}{\beta_1}\ln N\frac{2i}{N},
                                             &i=0,\dots,N/2,\\
             1-2(1-\lambda_x)(1-\frac{i}{N}),&i=N/2,\dots,N
            \end{cases}
    \end{align*}
    and $h_i=x_{i}-x_{i-1}$ the local mesh size.
    We have to estimate
    \begin{gather}\label{eq:1d}
      \int_0^1b(S-\hat S)' v,
    \end{gather}
    where $v$ is piecewise polynomial of degree $p=3$ and $\hat S$ some Lagrange
    interpolant of $S$ with $\hat S\in H_0^1(0,1)$. Later we will see that the
    estimates require some properties of the interior interpolation points that
    are fulfilled e.g. for the Gauß-Lobatto interpolation operator.

    Now, using \eqref{eq:representation} and setting $\eta=S-\hat S$ we can rewrite \eqref{eq:1d} as
    \begin{gather}\label{eq:1d'}
      \int_0^1 b(S-\hat S)'v
       = \sum_{i=1}^{N-1}\int_{x_{i-1}}^{x_{i+1}}b\eta'(v_i\phi_i+w_i\psi_{3,i})+
         \sum_{j=1}^N\int_{x_{j-1}}^{x_j}y_jb\eta'\chi_{2,j}
         +\int_{x_{N-1}}^1w_Nb\eta'\psi_{3,N}.
    \end{gather}
    In the two sums we will replace $b\eta'$ by
    \[
      b\eta'=b_i\tilde\eta_i'+(b-b_i)\eta'+b_i(\eta-\tilde\eta_i)'
    \]
    with constant $b_i=b(x_i)$ and $\tilde\eta_i$ defined in such a way that
    \begin{itemize}
      \item $\int\limits_{x_{i-1}}^{x_{i+1}}\tilde\eta_i'\phi_i=0$,\quad
            $\int\limits_{x_{i-1}}^{x_{i+1}}\tilde\eta_i'\psi_{3,i}=0$,\quad
            for $i\in\{1,\dots,N-1\}\setminus\{N/2\}$,
      \item $\int\limits_{x_{i-1}}^{x_{i}}\tilde\eta_i'\chi_{2,i}=0$ for $i=1,\dots,N$,
      \item $\norm{(\eta-\tilde\eta_i)'}{L_\infty(x_{i-1},x_{i+1})}$ is of order 4 in $h_i+h_{i+1}$
            (compared to $\norm{\eta'}{L_\infty(x_{i-1},x_{i+1})}$ being of order 3).
    \end{itemize}
    We will now show, that such an $\tilde\eta_i$ exists.
    It is well known that the interpolation error $S-\hat S=\eta$ can be represented as
    \[
     (S-\hat S)(x)=\frac{S^{(4)}(\xi(x))}{4!}(x-x_{i-1})(x-\alpha_i)(x-\beta_i)(x-x_i)
    \]
    if interpolated in $x_{i-1},\,\alpha_i,\,\beta_i$ and $x_i$, where $\alpha_i$ and $\beta_i$
    are the interior interpolation points. Consequently,
    \begin{gather}\label{eq:tilde_eta}
      (S-\hat S)(x)=\frac{S^{(4)}(x_i)}{4!}(x-x_{i-1})(x-\alpha_i)(x-\beta_i)(x-x_i)+\ord{h_i^5}
    \end{gather}
    on $[x_{i-1},x_i]$. Thus we set set
    \[
      \tilde\eta_i=\begin{cases}
                    \frac{S^{(4)}(x_i)}{4!}(x-x_{i-1})(x-\alpha_i    )(x-\beta_i    )(x-x_i    ),&x\in[x_{i-1},x_i],\\
                    \frac{S^{(4)}(x_i)}{4!}(x-x_{i  })(x-\alpha_{i+1})(x-\beta_{i+1})(x-x_{i+1}),&x\in[x_i,x_{i+1}].
                   \end{cases}
    \]
    By the choice of the symmetric interior interpolation points of the Gauß-Lobatto
    interpolation, our approximation $\tilde\eta_i$ is an even function on the three intervals 
    $[x_{i-1},x_{i+1}]$, $[x_{i-1},x_i]$ and $[x_i,x_{i+1}]$. 
    Therefore, $\tilde\eta_i'$ is an odd function on these intervals.
    Together with $\phi_i$ and $\psi_{3,i}$ being even on $[x_{i-1},x_{i+1}]$ for
    $i\in\{1,\dots,N-1\}\setminus\{N/2\}$ and $\chi_{2,i}$ being
    even on $[x_{i-1},x_i]$ for any $i$, we obtain the first two wanted properties.
    The last property is due to \eqref{eq:tilde_eta}.
    
    Thus \eqref{eq:1d'} can be rewritten as
    \begin{align}\label{eq:1d''}
      \int_0^1 b(S-\hat S)'v
       = \quad &\int_{x_{N/2-1}}^{x_{N/2+1}}b_{N/2}\tilde\eta_i'(v_{N/2}\phi_{N/2}+w_{N/2}\psi_{3,N/2})\notag\\
         +&\sum_{i=1}^{N-1}\int_{x_{i-1}}^{x_{i+1}}[(b-b_i)\eta'+b_i(\eta-\tilde\eta_i)'](v_i\phi_i+w_i\psi_{3,i})\notag \\
         +&\sum_{j=1}^N\int_{x_{j-1}}^{x_j}[(b-b_j)\eta'+b_j(\eta-\tilde\eta_j)']y_j\chi_{2,j}\notag\\
         +&\int_{x_{N-1}}^1w_Nb\eta'\psi_{3,N}
         =:I+II+III+IV.
    \end{align}
%

    \textbf{I}: For the first term of \eqref{eq:1d''} we obtain
    \[
      |I|=\left|\int_{x_{N/2-1}}^{x_{N/2+1}}b_{N/2}\tilde\eta_i'(v_{N/2}\phi_{N/2}+w_{N/2}\psi_{3,N/2})\right|
      \lesssim N^{-3}(x_{N/2+1}-x_{N/2-1})(|v_{N/2}|+|w_{N/2}|).
    \]
    A Cauchy-Schwarz inequality gives
    \begin{gather}\label{eq:v_N2:estimate}
      v_{N/2}=\int_0^{\lambda_x}v'
             \lesssim \norm{v'}{L_1(0,\lambda_x)}.
    \end{gather}
    For $w_{N/2}$ we recall
    \[
       w_{N/2} = \int_0^{\lambda_x}\widetilde L_3 v
               = \sum_{k=1}^{N/2}\frac{\int_{x_{k-1}}^{x_k}L_3^k(x)v(x)dx}
                                       {\int_{x_{k-1}}^{x_k} L_3^k(x) \psi_{3,k}(x)dx}.
    \]
    With $L_3^k$ being odd on $[x_{k-1},x_k]$ it holds
    \begin{align*}
      \int_{x_{k-1}}^{x_k}L_3^k(x)v(x)dx
      &=\int_{x_{k-1}}^{x_k}L_3^k(x)\frac{v(x)-v(x_k-(x-x_{k-1}))}{2}dx\\
      &=\frac{1}{2}\int_{x_{k-1}}^{x_k}L_3^k(x)\int_{x_k-(x-x_{k-1})}^xv'(t)dtdx\\
      &\leq \frac{1}{2}\norm{L_3^k}{L_1[x_{k-1},x_k]}\norm{v'}{L_1[x_{k-1},x_k]}.
    \end{align*}
    Thus we have for $w_{N/2}$
    \begin{gather}\label{eq:w_N2:estimate}
       |w_{N/2}|\lesssim \sum_{k=1}^{N/2}\frac{\norm{L_3^k}{L_1[x_{k-1},x_k]}}
                                              {\norm{L_3^k}{0,[x_{k-1},x_k]}^2}\norm{v'}{L_1[x_{k-1},x_k]}
              \lesssim\norm{v'}{L_1[0,\lambda_x]}.
    \end{gather}
    Combining the estimates for the two coefficients yields
    \begin{gather}\label{eq:rhs:1}
      |I|
        \lesssim N^{-4}\norm{v'}{L_1(0,\lambda_x)}
        \lesssim N^{-4}(\eps\ln N)^{1/2}\norm{v'}{0}
        \lesssim N^{-4}(\ln N)^{1/2}\enorm{v}.
    \end{gather}

    \textbf{II+III}: It holds with the interpolation properties of
    $b-b_i$, $\eta'$ and $(\eta-\tilde\eta_i)'$
    \begin{align*}
     (II+III)^2
      &\leq 2(II^2+III^2)
       \lesssim N^{-8}\left[ \norm{v}{0}^2
                            -\frac{1}{2}\sum_{i=1}^{N-1}v_i\left[h_iy_i+h_{i+1}y_{i+1}\right]
                            -\frac{1}{6}v_{N-1}w_Nh_N\right].
    \end{align*}
    The coefficients $v_i$, $y_i$ and $w_N$ can be bound by
    \begin{align*}
     |h_iy_i| &= |h_i N_2^iv|  = h_i\left|\frac{\int_{x_{i-1}}^{x_i}L_2^i v}{\int_{x_{i-1}}^{x_i}L_2^i \chi_{2,i}}\right|
               \lesssim \norm{v}{L_1(x_{i-1},x_i)}\\
     |w_N| &\leq |w_{N/2}|+\norm{v'}{L_1(\lambda_x,1)}
            \lesssim (\ln N)^{1/2}\enorm{v}+N\norm{v}{0}\\
     |v_i| &\lesssim\begin{cases}
                     (\ln N)^{1/2}\enorm{v},&i\leq N/2\\
                     N\norm{v}{L_1(x_{i-1},x_{i+1})},&j>N/2
                    \end{cases}
    \end{align*}
    where we have used \eqref{eq:w_N2:estimate} and an inverse inequality in the second line, 
    and a similar reasoning to \eqref{eq:v_N2:estimate} and an inverse inequality in the last line.
    Thus, we obtain
    \begin{align}
     (II+III)^2
      &\lesssim N^{-8}\bigg[ \norm{v}{0}^2
                            +(\ln N)^{1/2}\enorm{v}\norm{v}{L_1(0,x_{N/2+1})}
                            +\sum_{i=N/2+1}^{N-1}N\norm{v}{L_1(x_{i-1},x_{i+1})}^2\notag\\&\hspace*{2cm}
                            +\norm{v}{L_1(x_{N-2},x_{N})}((\ln N)^{1/2}\enorm{v}+N\norm{v}{0})\bigg]\notag\\
      &\lesssim N^{-8}\bigg[
                             (\ln N)^{1/2}\enorm{v}^2
                            +N^{1/2}\norm{v}{0}^2\bigg]
       \lesssim N^{-(8-1/2)}\enorm{v}^2.\label{eq:rhs:2h}
    \end{align}
    Therefore, we can conclude
    \begin{gather}\label{eq:rhs:2}
      |II+III|
        \lesssim N^{-(4-1/4)}\enorm{v}.
    \end{gather}
    
    \textbf{IV}: Finally, integration by parts, the bound on $|w_N|$ and the interpolation properties of $\eta$ give
    \begin{gather*}
      IV
        =-\int_{x_{N-1}}^1w_N b'\eta\psi_{3,N}
         -\int_{x_{N-1}}^1w_N b \eta\psi_{3,N}'
        \lesssim N^{-4}(\enorm{v}+\norm{w_N \psi_{3,N}'}{L_1(x_{N-1},1)}).
    \end{gather*}
    For $\norm{w_N \psi_{3,N}'}{L_1(x_{N-1},1)}$ an inverse inequality gives
    \begin{align*}
     \norm{w_N \psi_{3,N}'}{L_1(x_{N-1},1)}
     &\lesssim N\norm{w_N \psi_{3,N}}{L_1(x_{N-1},1)}
      \lesssim N^{1/2}\norm{w_N \psi_{3,N}}{0,(x_{N-1},1)}\\
     &\lesssim N^{1/2}\left(\norm{v}{0}^2
                          -\frac{1}{2}\sum_{i=1}^{N-1}v_i\left[h_iy_i+h_{i+1}y_{i+1}\right]
                          -\frac{1}{6}v_{N-1}w_Nh_N\right)^{1/2}\\
     &\lesssim N^{1/2}N^{1/4}\enorm{v},
    \end{align*}
    where the estimation of the scalar products in \eqref{eq:rhs:2h} was used.
    Together we obtain
    \begin{gather}\label{eq:rhs:4}
      |IV|\lesssim N^{-(4-3/4)}\enorm{v}
    \end{gather}         
    Combining \eqref{eq:rhs:1}, \eqref{eq:rhs:2} and \eqref{eq:rhs:4} finishes the proof.
  \end{proof}

  \begin{lem}\label{lem:4_2d}
    It holds
    \begin{gather}\label{eq:analysis:7}
      |(b\cdot\grad(S-\hat S),v)|
      \lesssim N^{-(p+1/4)}\enorm{v}.
    \end{gather}   
  \end{lem}
  \begin{proof}
    For any odd polynomial degree $p$ larger than three, we simply extend the approach
    of Lemma~\ref{lem:4_1d}.
    On each interval $[x_{i-1},x_i]$ we add even-order bubble functions $\chi_{2k,i}$,
    $k=2,\dots,(p-1)/2$.
    They are defined on $[0,1]$ by
    \[
      \hat\chi_{2k}(t):=\frac{1-L_{2k}(2t-1)}{2}
    \]
    and mapped linearly onto $[x_{i-1},x_i]$.
    On each double interval $[x_{i-1},x_{i+1}]$ we add piecewise
    polynomial bubble functions $\psi_{2k+1,i}$, $k=2,\dots,(p-1)/2$,
    defined on the reference interval $[-1,1]$ by
    \[
      \hat\psi_{2k+1}(t):=\frac{L_1(2|t|-1)-L_{2k+1}(2|t|-1)}{2}
    \]
    and mapped by $F_i$. Figure~\ref{fig:basis2}
    \begin{figure}
      \begin{center}
       \includegraphics[width=0.32\textwidth]{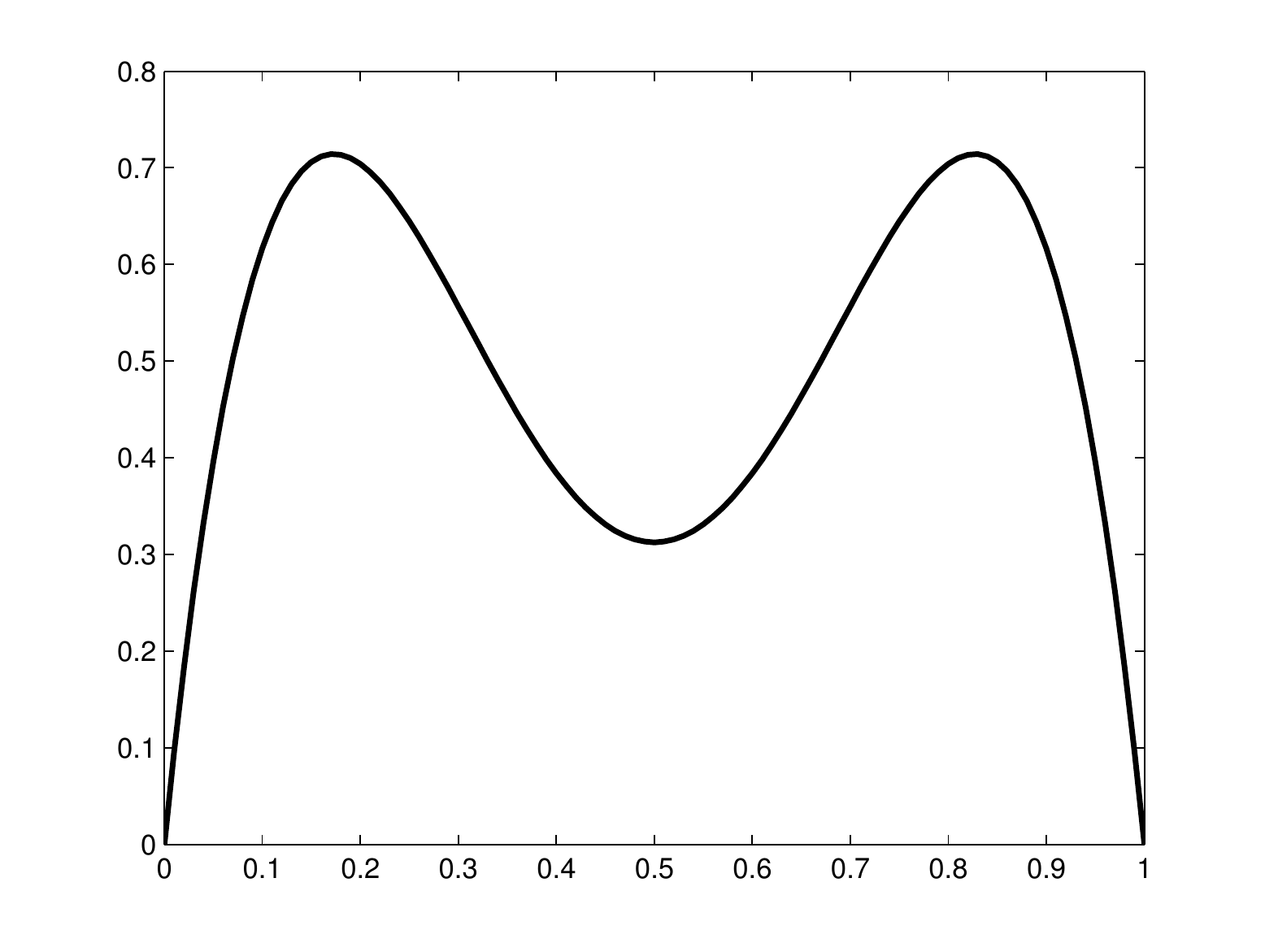}
       \includegraphics[width=0.32\textwidth]{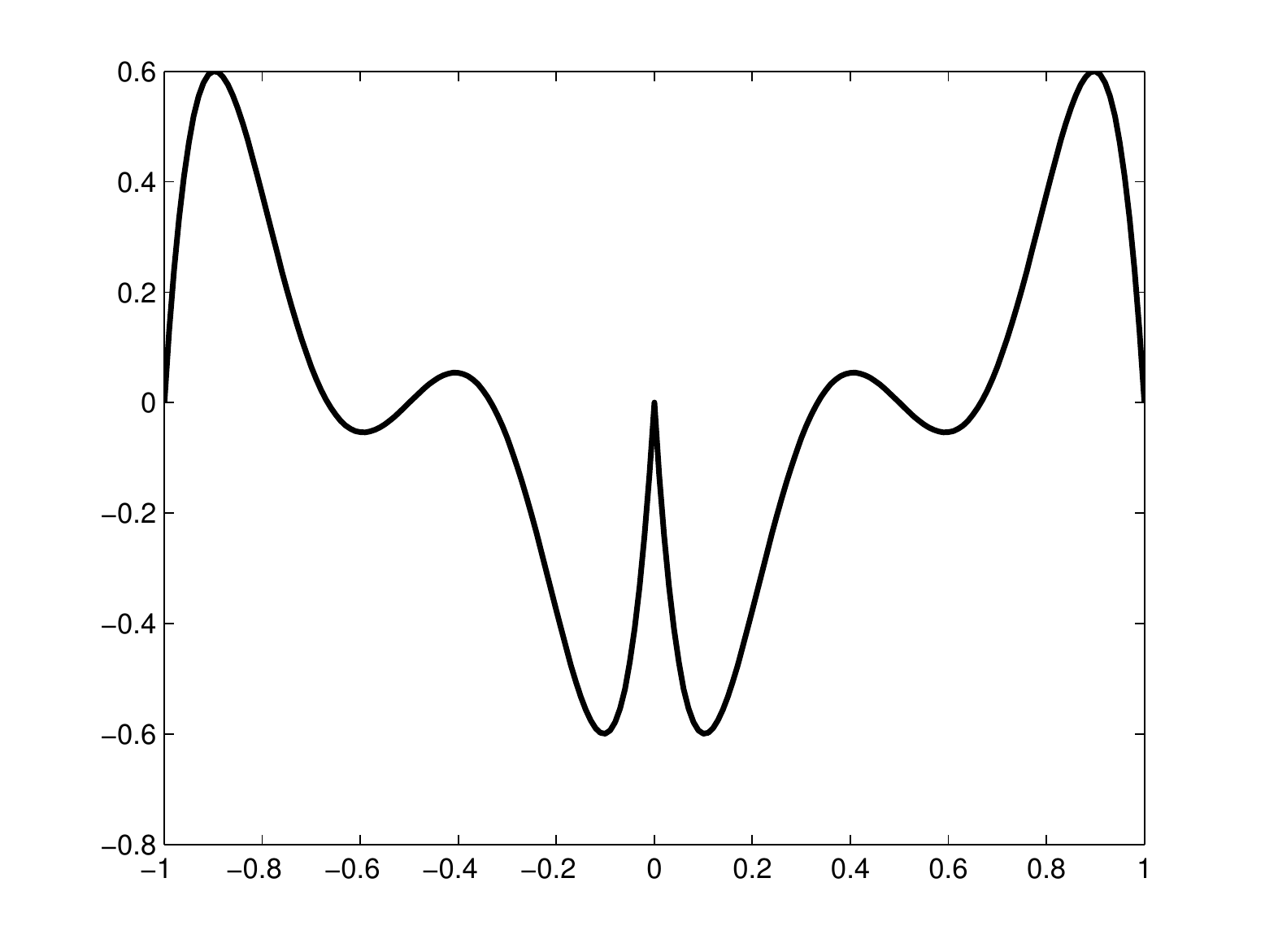}
      \end{center}
      \caption{Additional basis functions $\hat\chi_4$ (left) and $\hat\psi_5$ (right) on their domains of support
               \label{fig:basis2}}
    \end{figure}
    shows in the case of $p=5$ the two additional functions. Thus we obtain the representation
    \[
      v= \sum_{i=1}^{N-1}v_i\phi_i
        +\sum_{k=1}^{(p-1)/2}\sum_{i=1}^{N-1}w_i^{2k+1}\psi_{2k+1,i}
        +\sum_{k=1}^{(p-1)/2}\sum_{i=1}^N y_j^{2k}\chi_{2k,j}
        +\sum_{k=1}^{(p-1)/2}w_n\psi_{2k+1,n}.
    \]
    The new coefficients can be defined by using the degrees of freedom
    \begin{align*}
       N_{2k}^jv  &:=\frac{\int_{x_{j-1}}^{x_j} L_{2k}^j(x) v(x)dx}
                          {\int_{x_{j-1}}^{x_j} L_{2k}^j(x) \chi_{2k,j}(x)dx},    &&\,j=1,\dots,N,\\
       N_{2k+1}^iv&:=\frac{\int_{x_{i-1}}^{x_i} L_{2k+1}^i(x) v(x)dx}
                          {\int_{x_{i-1}}^{x_i} L_{2k+1}^i(x) \psi_{2k+1,i}(x)dx},&&\,i=1,\dots,N.
    \end{align*}
    If we compare the new basis functions with the old ones $\chi_{2,i}$ and $\psi_{3,i}$,
    we notice a very similar behaviour. Thus, the same analytical steps can be applied
    and it follows for the convective term in $S$ and any odd degree $p$
    \begin{gather}\label{eq:analysis:6}
      \int_0^1 b(S-\hat S)'v\lesssim N^{-(p+1/4)}\enorm{v}.
    \end{gather}
  
    The extension to the two-dimensional problem is fairly easy. By the tensor-product structure
    of our problem, the mesh and the definitions of the norms, we obtain
    immediately from \eqref{eq:analysis:6}
    \[
      (b\cdot\grad(S-\hat S),v)
      =(b_1(S-\hat S)_x,v)+
       (b_2(S-\hat S)_y,v)
      \lesssim N^{-(p+1/4)}\enorm{v}.\qedhere
    \]
  \end{proof}

  Consequently, by combining \eqref{eq:analysis:start} and Lemmas~\ref{lem:1}--\ref{lem:3} and \ref{lem:4_2d}
  we have the main result of this paper.
  
  \begin{thm}
   For the Galerkin solution $u^N$ of a finite element method of odd degree $p$ holds
   \[
     \enorm{u^N-J^Nu}\lesssim (N^{-1}\ln N)^{p+1}+N^{-(p+1/4)}
   \]
   where $J^N$ is either the vertex-edge-cell interpolation operator $\pi^N$ or
   the Gauß-Lobatto interpolation operator $I^N$.
  \end{thm}
  \begin{proof}
   By combining the previous Lemmas we have the main result for the Gauß-Lobatto
   interpolation operator immediately. For the vertex-edge-cell interpolation operator
   $\pi^N$ we use the identity \eqref{eq:identity} and the ideas presented at the end of
   the proof of Lemma~\ref{lem:2}.
  \end{proof}

  \begin{cor}
   With a suitable postprocessing operator $P^N$ that maps the piecewise
   $\QS_p$-solution into a piecewise $\QS_{p+1}$-solution on a macro-mesh, a superconvergence
   property of the numerical solution $P^Nu^N$
   \[
    \enorm{P^Nu^N-u}\lesssim (N^{-1}\ln N)^{p+1}+N^{-(p+1/4)}
   \]
   can be deduced easily. For details and examples of suitable operators, see e.g. \cite{Fr12}.
  \end{cor}

  \bibliographystyle{plain}
  \bibliography{lit}

\end{document}